\documentclass[11pt,a4paper]{article}

\usepackage[T1]{fontenc}
\usepackage[utf8]{inputenc}
\usepackage{amsmath,amssymb,amsthm,mathtools,cite}
\usepackage{enumitem}
\usepackage{tikz}
\usetikzlibrary{positioning,fit,calc,arrows.meta,backgrounds,shapes.geometric}
\usepackage{microtype}
\usepackage{hyperref}
\usepackage{fullpage}
\hypersetup{
  colorlinks=true,
  linkcolor=blue,
  citecolor=blue,
  urlcolor=blue,
  pdftitle={On Alon's problem concerning the difference between connected domination number and domination number},
  pdfauthor={C. Li and B. Zhou},
  pdfsubject={connected domination},
  pdfkeywords={domination number, connected domination number, minimum degree, probabilistic method}
}

\usepackage{authblk}

\newtheorem{theorem}{Theorem}[section]
\newtheorem{corollary}[theorem]{Corollary}
\newtheorem{proposition}[theorem]{Proposition}
\newtheorem{lemma}[theorem]{Lemma}
\theoremstyle{definition}
\newtheorem{problem}[theorem]{Problem}
\theoremstyle{remark}

\newcommand{\cF}{\mathcal F}
\DeclareMathOperator{\Bin}{Bin}
\DeclareMathOperator{\cov}{cov}
\allowdisplaybreaks

\begin{document}

\title{\bfseries
On Alon's problem concerning the difference between connected domination number and domination number}

\author[1]{ Chengli Li%
  \thanks{Email: lichengli0130@126.com}}
\author[2]{ Bo Zhou%
  \thanks{Email: zhoubo@m.scnu.edu.cn}}

\affil[1]{\footnotesize
School of Mathematical Sciences, East China Normal University,
Shanghai 200241, China}

\affil[2]{\footnotesize
School of Mathematical Sciences, South China Normal University, Guangzhou 510631, China}

\date{}
\maketitle
	
\begin{abstract}
For a connected graph $G$, let $\gamma(G)$ and $\gamma_c(G)$
denote its domination number and connected domination number, respectively. Let $M(n,k)$ be the maximum of
$\gamma_c(G)-\gamma(G)$ over all connected $n$-vertex graphs
of minimum degree at least $k$.
Alon proved that 
\[
2\left\lfloor\frac{n}{k+1}\right\rfloor-O(1)\le M(n,k)<\frac{n}{k+1}
\bigl(\log\lceil\log(k+1)\rceil+3\bigr).
\]
He proposed a problem to determine or estimate   $M(n,k)$ for $n-1\ge k\ge 3$, and particularly remarked that it would be interesting to close the $\log\log(k+1)$ gap between the upper and lower bounds and decide whether or not $M(n,k)=\Theta\bigl(\frac{n}{k+1}\bigr)$.  
We give an asymptotic answer to Alon's
problem for sufficiently large $k$.
More precisely, $M(k+1,k)=0$. For $n>k+1$, let
\[
\nu=\frac{n}{k+1}
\text{ and }
\Phi(x)=
\frac{1}{x\log\!\frac{x}{x-1}} \ (\text{for }x>1).
\]
As $k\to\infty$, uniformly over all integers $n>k+1$, 
\[
M(n,k)=
\big(\Phi(\nu)+o(1)\big)
\nu\log\log(k+1).
\]

\smallskip
\noindent\textbf{Keywords:} domination number, connected domination number, minimum degree, probabilistic method

\noindent\textbf{2020 Mathematics Subject Classification:} 05C69, 05D40, 05C35
\end{abstract}

\section{Introduction}

All graphs are finite, simple and undirected. For a graph $G$, $V(G)$ and $E(G)$ denote its vertex set and edge set, respectively, and $|V(G)|$ is its order. For $v\in V(G)$, let $N_G(v)$ be the set of neighbors of $v$ and let $N_G[v]=N_G(v)\cup\{v\}$ be its closed neighborhood. The degree of $v$ in $G$ is  $d_G(v)=|N_G(v)|$. Let  $\delta(G)$ be the  minimum degree of $G$. For $S\subseteq V(G)$,  $G[S]$ denotes the subgraph of $G$ induced by $S$.  A clique is a set of pairwise adjacent vertices. Two disjoint vertex sets $X,Y\subseteq V(G)$ are anticomplete if no edge has one endpoint in $X$ and the other in $Y$. 
For disjoint sets $X,Y\subseteq V(G)$, we say that $X$ dominates $Y$ if every vertex of $Y$ has a neighbor in $X$.

Let $G$ be a connected graph. A set $D\subseteq V(G)$ is a dominating set if every vertex in $V(G)\setminus D$ has a neighbor in $D$. Moreover, if $G[D]$ is connected, then $D$ is a connected dominating set. The minimum cardinality of a dominating set of $G$ is the domination number of $G$, written as $\gamma(G)$. 
The minimum cardinality of a connected dominating set of $G$ is the connected domination number of $G$, written as $\gamma_c(G)$. By the classical complementarity between connected dominating sets and spanning trees, 
$\gamma_c(G)$ is precisely the minimum possible number of non-leaves in a spanning tree of $G$. 
Thus, studying the maximum number of leaves in a spanning tree is equivalent to studying the connected domination number.
Connected domination was introduced by Sampathkumar and Walikar \cite{SaWa}, and has since been investigated extensively \cite{HedetniemiLaskar,CaroWestYuster,GuhaKhuller,HaynesHedetniemiSlater}. The maximum-leaf formulation has been studied from extremal and algorithmic viewpoints  \cite{Bonsma,GriggsKleitmanShastri,KleitmanWest,GriggsWu,DingJohnsonSeymour,GalbiatiMaffioliMorzenti,SolisOba,BonsmaZickfeld}.

For integers $n$ and $k$ with $n-1\ge k>1$, define $M(n,k)$ to be  the maximum of $\gamma_c(G)-\gamma(G)$ over all connected graphs $G$ with $n$ vertices and minimum degree at least $k$.

All logarithms in this paper are natural.

\begin{theorem}[Alon \cite{Alon2023}]\label{thm:alon}
For $n-1\ge k>1$, 
\[
2\left\lfloor \frac{n}{k+1}\right\rfloor-O(1)\le M(n,k)<\frac{n}{k+1}
\bigl(\log\lceil\log(k+1)\rceil+3\bigr).
\]
\end{theorem}

Alon concluded his paper with the following problem.

\begin{problem}[Alon~\cite{Alon2023}]\label{prob:alon}
Determine or estimate $M(n,k)$ for
$k\ge 3$.
\end{problem}

Alon~\cite{Alon2023} also remarked that 
it would be interesting to close the $\log\log(k+1)$ gap between the upper and lower
bounds for $M(n,k)$ and determine whether
$M(n,k)=\Theta\!\left(\frac{n}{k+1}\right)$.

The boundary case $n=k+1$ is immediate. Indeed, if
$|V(G)|=k+1$ and $\delta(G)\ge k$, then
$G=K_{k+1}$. Hence $\gamma(G)=\gamma_c(G)=1$, and therefore
$M(k+1,k)=0$. We henceforth assume that $n>k+1$.

For $x>1$, define
\[
\Phi(x)=\frac{1}{x\log\frac{x}{x-1}}.
\]

\begin{lemma}\label{lem:Phi}
The function $\Phi$ is strictly increasing on $(1,\infty)$, $0<\Phi(x)<1$, 
$\lim_{x\to 1^+}\Phi(x)=0$ and 
$\lim_{x\to\infty}\Phi(x)=1$. 
\end{lemma}

\begin{proof}
Let 
$g(x)=\frac{1}{\Phi(x)}$. Then $g(x)=x\log \frac{x}{x-1}$.

As $\log(1+t)<t$ for $t>0$,
$
g'(x)=\log\frac{x}{x-1}-\frac{1}{x-1}<0$,
so $g$ is strictly decreasing and $\Phi(x)$ is strictly increasing.

Note that $\log(1+t)>\frac{t}{1+t}$ for $t>0$. Taking $t=\frac{1}{x-1}$ gives $g(x)>1$, so $0<\Phi(x)<1$.

As $x\to 1^+$, we have $\frac{x}{x-1}\to\infty$ and then
$g(x)\to\infty$.
Taking $t=\frac{1}{x-1}$ in the inequalities
$\frac{t}{1+t}<\log(1+t)<t$, we have 
$
1<g(x)
=x\log\left(1+\frac{1}{x-1}\right)
<\frac{x}{x-1}$.
As $x\to\infty$, $\frac{x}{x-1}\to1$, so  $g(x)\to 1$. Therefore,
$\lim_{x\to1^+}\Phi(x)=0$
and
$\lim_{x\to\infty}\Phi(x)=1$. 
\end{proof}

The main results are listed below.

\begin{theorem}\label{thm:full-range}
For every $\eta>0$, there exists an integer $K=K(\eta)$ such
that, for every integer $k\ge K$ and every integer $n>k+1$ with 
$\nu=\frac{n}{k+1}$,
\[
\left|
\frac{M(n,k)}
{\nu\log\log(k+1)}
-
\Phi(\nu)
\right|<\eta.
\]
\end{theorem}

By Theorem \ref{thm:full-range}, 
 as $k\to\infty$, uniformly over all integers $n>k+1$,
\[
M(n,k)=
\bigl(\Phi(\nu)+o(1)\bigr)
\nu\log\log(k+1).
\]
This provides an asymptotic estimate for $M(n,k)$, and hence an asymptotic answer to
Problem~\ref{prob:alon}.

\begin{corollary}\label{cor:theta}
For every fixed $C>1$, there exists an integer
$k_0=k_0(C)$ such that, for all integers $k\ge k_0$ and
$n\ge C(k+1)$,
\[
\frac{\Phi(C)}{2}\cdot
\frac{n}{k+1}\log\log(k+1)
<
M(n,k)
<
\frac{3}{2}\cdot
\frac{n}{k+1}\log\log(k+1).
\]
\end{corollary}

From Corollary~\ref{cor:theta},
$
M(n,k)=
\Theta\!\left(
\frac{n}{k+1}\log\log(k+1)
\right)$.

\begin{corollary}\label{cor:large-ratio}
If $k\to\infty$ and
$
\frac{n}{k+1}\to\infty,
$
then
$
M(n,k)\sim
\frac{n}{k+1}\log\log(k+1).
$
\end{corollary}

The proof  of Theorem \ref{thm:full-range} relies on the following theorem, which is of independent interest though it is covered by Theorem \ref{thm:full-range}.

\begin{theorem}\label{thm:main}
For every $\sigma>0$ and every $\varepsilon>0$, there exists an
integer $k_0=k_0(\sigma,\varepsilon)$ such that, for every
integer $k\ge k_0$ and every integer $n$ satisfying
$
\nu=\frac{n}{k+1}\ge1+\sigma,
$
we have
\[
-\varepsilon< \frac{M(n,k)}{\nu\log\log(k+1)}-\Phi(\nu)<
\varepsilon.
\]
\end{theorem}



%
%
%

\section{Proof of the upper bound in Theorem \ref{thm:main}}

We use the following classical domination estimate. It follows from the standard set-cover bound, see \cite{Lovasz,AlonSpencer}.

\begin{lemma}\label{lem:domination-standard}
Every $n$-vertex graph of minimum degree at least $k$ satisfies
\[
\gamma(G)\le\frac{n}{k+1}\bigl(\log(k+1)+1\bigr).
\]
\end{lemma}

Alon established an upper bound for $M(n,k)$ in \cite[Lemma 3.1]{Alon2023}. We need the following upper bound for $M(n,k)$.

\begin{proposition}\label{prop:upper}
Let $k\ge3$, let $n>k+1$, and 
let $\nu=\frac{n}{k+1}>1$.
Then
\[
M(n,k)\le
\frac{\log\bigl(\nu\log(k+1)\bigr)}
{\log \frac{\nu}{\nu-1}}
+2\nu+1.
\]
\end{proposition}

To prove this result, we borrow the techniques from the proof of \cite[Lemma 3.1]{Alon2023}, and particularly, we need a lemma, which is a standard observation,  see, e.g., the proof of \cite[Lemma 3.1]{Alon2023}.


\begin{lemma}\label{lem:connect-components}
Let $G$ be a connected graph, and let $S$ be a dominating set such that $G[S]$ has $x$ connected components. Then $S$ can be enlarged to a connected dominating set by adding at most $2(x-1)$ vertices.
\end{lemma}

%
%

\begin{proof}[Proof of Proposition \ref{prop:upper}]
Let $G$ be a connected $n$-vertex graph with $\delta(G)\ge k$. Let $S$ be a minimum dominating set, and let $x_0$ be the number of components of $G[S]$. Lemma~\ref{lem:domination-standard} gives
\[
x_0\le |S|=\gamma(G)
\le\nu\bigl(\log(k+1)+1\bigr).
\]

Starting from $S$, add vertices when the current dominating set induces more than $\nu+1$ components. Suppose the current number of components is $x>\nu+1$. Choose one representative vertex from each component. The sum of the sizes of the closed neighborhoods of these $x$ representatives is at least $(k+1)x$. Counting the pairs
$(u,\text{representative }v)$ with $u\in N_G[v]$
and averaging over the $n$ possible vertices $u$, we find a vertex $w$ belonging to at least
$
\left\lceil \frac{(k+1)x}{n} \right\rceil
=\left\lceil \frac{x}{\nu} \right\rceil
$
of the chosen closed neighborhoods.
Because $x>\nu+1$, this number is at least two. 

If $w$ were already in the current dominating set, then any two
representatives whose closed neighborhoods contain $w$ would be
connected through $w$ in the subgraph induced by that set. Since the
representatives belong to distinct components, this is impossible.
Thus $w$ is not already in the current dominating set.
Adding $w$ therefore merges at least $\lceil \frac{x}{\nu}\rceil$ old components into one. If $x'$ denotes the new number of components, then
\[
x'\le x-\left\lceil\frac{x}{\nu}\right\rceil+1
\le\left(1-\frac{1}{\nu}\right)x+1.
\]
Subtracting $\nu$ gives the recurrence
\begin{equation}\label{eq:translated-recurrence}
x'-\nu\le
\left(1-\frac{1}{\nu}\right)(x-\nu).
\end{equation}

Let $x_i$ be the number of components after $i$ such additions, and
write $y_i=x_i-\nu$. 
If the procedure performs at least $i$ additions, then
\eqref{eq:translated-recurrence} applies successively for
$j=0,\ldots,i-1$, and hence
\[
y_i\le
\left(1-\frac{1}{\nu}\right)^i y_0.
\]
Let
$
\lambda_\nu
=-\log\left(1-\frac{1}{\nu}\right)
=\log\frac{\nu}{\nu-1}>0
$
and let
$
t_0=
\left\lceil
\frac{\log\max\{x_0-\nu,1\}}{\lambda_\nu}
\right\rceil.
$
If $x_0\le\nu+1$, then $t_0=0$, and no addition is needed.
Suppose, therefore, that $x_0>\nu+1$. If the procedure terminates
before $t_0$ additions, then the current dominating set already
induces at most $\nu+1$ components. Otherwise, the procedure can be
performed $t_0$ times, and
\[
x_{t_0}-\nu
\le
\left(1-\frac{1}{\nu}\right)^{t_0}(x_0-\nu)
=
e^{-\lambda_\nu t_0}(x_0-\nu)
\le 1,
\]
where the last inequality follows from the definition of $t_0$.
Consequently, $x_{t_0}\le\nu+1$. Thus, after at most $t_0$
additions, the current dominating set induces at most $\nu+1$
components.

By Lemma~\ref{lem:connect-components}, it can then be extended to a
connected dominating set by adding at most
$2\nu$
further vertices. Therefore,
$
\gamma_c(G)-\gamma(G)\le t_0+2\nu.
$

Finally, since
$
x_0-\nu\le\nu\log(k+1)
$
and $\nu\log(k+1)>1$ for $k\ge3$, we have
$
\max\{x_0-\nu,1\}\le\nu\log(k+1).
$
It follows that
\[
t_0
\le
\frac{\log\bigl(\nu\log(k+1)\bigr)}{\lambda_\nu}+1.
\]
Since $G$ was arbitrary, taking the maximum over all such graphs gives
\[
M(n,k)
\le
\frac{\log\bigl(\nu\log(k+1)\bigr)}
{\log \frac{\nu}{\nu-1}}
+2\nu+1.
\qedhere
\]
\end{proof}


\begin{proof}[Proof of the upper bound in Theorem \ref{thm:main}]
Let $L_k=\log\log(k+1)$. Choose a constant $A>1+\sigma$ so large that
$1-\Phi(A)<\frac{\varepsilon}{4}$. 

First suppose $\nu\le A$. Proposition~\ref{prop:upper} gives
\[
\frac{M(n,k)}{\nu L_k}
\le
\Phi(\nu)
+\frac{\Phi(\nu)\log\nu+2+\frac{1}{\nu}}{L_k}.
\]
Because $0<\Phi(\nu)<1$ and $\nu\le A$, the second term is at most
\[
\frac{\log A+2+\frac{1}{1+\sigma}}{L_k},
\]
which is smaller than $\varepsilon$ when $k$ is large enough, where the lower bound on $k$ depends only on $\sigma$ and $\varepsilon$. That is,  $\frac{M(n,k)}{\nu\log\log(k+1)}-\Phi(\nu)<
\varepsilon$.

Now suppose $\nu>A$. Theorem~\ref{thm:alon} gives
\[
\frac{M(n,k)}{\nu L_k}
<
\frac{\log\lceil\log(k+1)\rceil+3}{L_k}.
\]
The last ratio is at most $1+\frac{\varepsilon}{4}$ for all sufficiently large $k$. Since $\Phi$ is increasing,
$\Phi(\nu)\ge\Phi(A)>1-\frac{\varepsilon}{4}$. Thus 
$\frac{M(n,k)}{\nu\log\log(k+1)}-\Phi(\nu)<1+\frac{\varepsilon}{4}-(1-\frac{\varepsilon}{4})<
\varepsilon$, as desired.
\end{proof}

\section{Two probabilistic ingredients}

For an integer $N\ge0$ and a real number $p\in[0,1]$, the notation $X\sim\Bin(N,p)$ means that $X$ has the binomial distribution with parameters $N$ and $p$. Equivalently, $X$ counts the number of successes in $N$ independent Bernoulli trials, each having success probability $p$. Thus
$\mathbb EX=Np$. 
We use the following standard Chernoff estimates; see, for example, \cite{AlonSpencer,JansonLuczakRucinski}.

\begin{lemma}[Chernoff bounds]\label{lem:chernoff}
Let $X\sim\Bin(N,p)$ and let $\mu=Np$. For $0<\eta<1$,
\[
\Pr\bigl(|X-\mu|>\eta\mu\bigr)
\le2\exp\left(-\frac{\eta^2\mu}{3}\right).
\]
Also, for $0<\theta<1$,
\[
\Pr\bigl(X\le(1-\theta)\mu\bigr)
\le\exp\left(-\frac{\theta^2\mu}{2}\right).
\]
\end{lemma}

For a positive integer \(N\), write \([N]=\{1,\ldots,N\}\).
An indexed set system
$\cF=(U;F_1,\ldots,F_N)$
consists of a ground set \(U\) and subsets \(F_j\subseteq U\), indexed
by \(j\in[N]\). If \(F_i=F_j\) with \(i\neq j\),
then the two occurrences are regarded as distinct indexed members.
Denote by
\[
d_{\cF}(u)=|\{j:u\in F_j\}|,
\qquad
\cov(\cF)=
\min\left\{
|J|:
J\subseteq[N],\
U=\bigcup_{j\in J}F_j
\right\},
\]
with the convention that \(\cov(\cF)=\infty\) if no such
\(J\) exists.

The covering-number estimate in the next lemma is closely related to a result of Vercellis~\cite[Theorem~3.1]{Vercellis}, which implies that for every fixed \(p\in(0,1)\), if \(\cF=(U;F_1,\ldots,F_N)\) has \(|U|=N\) and, independently for all \((u,j)\in U\times[N]\), the element \(u\) is included in \(F_j\) with probability \(p\), then, almost surely,
$
\cov(\cF)\sim
\frac{\log N}{-\log(1-p)}
$
as \(N\to\infty\).
Since our application requires a version with additional properties, we give a complete proof.

\begin{lemma}\label{lem:set-system}
Fix constants
$0<p_-<p_+<1,$ and 
$0<\eta<1$. 
There exists $N_0=N_0(p_-,\eta)$ such that, for every integer $N\ge N_0$ and every $p\in[p_-,p_+]$, there is an indexed set system
$\cF=(U;F_1,\ldots,F_N),$ with
$|U|=N,$ $D=pN$ and the following properties:
\begin{enumerate}[label=\textup{(\roman*)},leftmargin=2.2em]
\item $|F_j|\le(1+\eta)D$ for every $j$;
\item $(1-\eta)D\le d_{\cF}(u)\le(1+\eta)D$ for every $u\in U$;
\item
$
\cov(\cF)>
\frac{(1-\eta)\log N}{-\log(1-p)}.
$
\end{enumerate}
\end{lemma}

\begin{proof}
Let $U$ be a fixed set of $N$ elements. For every pair
$(u,j)\in U\times[N]$, independently decide that $u\in F_j$
with probability $p$. Let $I_{u,j}$ be the indicator of the event
$u\in F_j$. Thus the variables $I_{u,j}$ are mutually independent
and satisfy
$\Pr(I_{u,j}=1)=p$ and
$\Pr(I_{u,j}=0)=1-p$. 

Fix an index $j\in[N]$. By construction,
$|F_j|=\sum_{u\in U} I_{u,j}$. 
For this fixed $j$, the $N$ summands are independent Bernoulli
random variables with success probability $p$. Hence
$|F_j|\sim\Bin(N,p)$
and
$\mathbb E|F_j|=Np$.  By Lemma~\ref{lem:chernoff}, we have
\[
\Pr\bigl(|F_j|>(1+\eta)D\bigr)
\le
\Pr\bigl(\lvert |F_j|-D\rvert>\eta D\bigr)
\le
2\exp\left(-\frac{\eta^2D}{3}\right)
\le
2\exp\left(-\frac{\eta^2p_-N}{3}\right),
\]
where the last inequality follows from $D=pN\ge p_-N$.
There are exactly $N$ possible indices $j$. Therefore, by the
union bound,
\[
\Pr\bigl(\textup{(i) fails}\bigr)
\le
2N\exp\left(-\frac{\eta^2p_-N}{3}\right).
\]

Fix an element $u\in U$. Its degree in the indexed set system is
$d_{\cF}(u)=\sum_{j=1}^{N} I_{u,j}$. 
For this fixed $u$, the $N$ summands are again independent
Bernoulli random variables with success probability $p$.
Consequently,
$d_{\cF}(u)\sim\Bin(N,p)$
and $\mathbb E d_{\cF}(u)=Np=D$. 
The element $u$ violates the inequalities in \textup{(ii)} only if
$d_{\cF}(u)<(1-\eta)D$ or
$d_{\cF}(u)>(1+\eta)D$.
Thus the failure event for this particular element is contained in
$\bigl\{\lvert d_{\cF}(u)-D\rvert>\eta D\bigr\}$. 
Applying Lemma~\ref{lem:chernoff}, we obtain
\[
\Pr\bigl(u\text{ violates \textup{(ii)}}\bigr)
\le
2\exp\left(-\frac{\eta^2D}{3}\right)
\le
2\exp\left(-\frac{\eta^2p_-N}{3}\right).
\]
Taking the union bound over all $N$ elements of $U$ gives
\[
\Pr\bigl(\textup{(ii) fails}\bigr)
\le
2N\exp\left(-\frac{\eta^2p_-N}{3}\right).
\]
Combining the estimates for \textup{(i)} and \textup{(ii)}, we have
\begin{equation}\label{eq:set-degree-failure}
\Pr\bigl(\textup{(i) or \textup{(ii)} fails}\bigr)
\le
4N\exp\left(-\frac{\eta^2p_-N}{3}\right).
\end{equation}
Notice that the union bound does not require the failure events
corresponding to different indices or different elements to be
independent.

Let
$\lambda=-\log(1-p),$ and
$q_0=\left\lfloor
\frac{(1-\eta)\log N}{\lambda}
\right\rfloor$. 
Fix a set of indices $J\subseteq[N]$ with
$|J|=r\le q_0$. 
For each fixed $u\in U$, the events
$\{u\in F_j\}$, $j\in J$, are mutually independent. Therefore
\[
\Pr\left(
u\notin\bigcup_{j\in J}F_j
\right)
=(1-p)^r.
\]
Since
$r\le q_0
\le
\frac{(1-\eta)\log N}{\lambda},$
we have
\[
(1-p)^r=
\exp(-\lambda r)
\ge
\exp\bigl(-(1-\eta)\log N\bigr)
=N^{-(1-\eta)}.
\]
For distinct elements $u\in U$, the events
$\left\{
u\in\bigcup_{j\in J}F_j
\right\}$
are independent. It follows that
\[
\Pr\left(
U=\bigcup_{j\in J}F_j
\right)=
\left(1-(1-p)^r\right)^N\le
\exp\bigl(-N(1-p)^r\bigr)\le \exp(-N^\eta),
\]
where the first inequality follows from $1-x\le e^{-x}$, and the
second follows from
$N(1-p)^r\ge N\cdot N^{-(1-\eta)}=N^\eta$. 

Let
$\lambda_-=-\log(1-p_-)>0$. 
Since $p\ge p_-$, we have $\lambda\ge\lambda_-$, and hence
$q_0
\le
\frac{(1-\eta)\log N}{\lambda_-}$. 
After increasing $N_0$ if necessary, we may assume that $q_0<N$.
The number of subsets $J\subseteq[N]$ with $|J|\le q_0$ is at most
\[
\sum_{r=0}^{q_0}\binom Nr
\le
(q_0+1)N^{q_0}.
\]
Since
$q_0\le c_0\log N,$ where
$c_0=\frac{1-\eta}{\lambda_-},$
and $q_0+1\le N$ for all sufficiently large $N$, we have
\[
\begin{aligned}
(q_0+1)N^{q_0}
=\exp\bigl(\log(q_0+1)+q_0\log N\bigr)
\le
\exp\bigl(\log N+c_0(\log N)^2\bigr)
\le
\exp\bigl(C(\log N)^2\bigr)
\end{aligned}
\]
for a constant $C=C(p_-,\eta)$.

A further union bound over all possible choices of $J$ gives
\begin{equation}\label{eq:set-cover-failure}
\Pr\bigl(
\text{some family of at most $q_0$ sets covers $U$}
\bigr)
\le
\exp\bigl(C(\log N)^2-N^\eta\bigr).
\end{equation}

The right-hand sides of
\eqref{eq:set-degree-failure} and
\eqref{eq:set-cover-failure} tend to $0$ as $N\to\infty$.
All constants in these estimates depend only on
$p_-$ and $\eta$, and not on the particular value of
$p\in[p_-,p_+]$. We may therefore choose
$N_0=N_0(p_-,\eta)$
so that, for every $N\ge N_0$ and every
$p\in[p_-,p_+]$, the sum of the two failure probabilities is
smaller than $1$.

Consequently, there is an outcome for which
\textup{(i)} and \textup{(ii)} hold and no family of at most
$q_0$ indexed sets covers $U$. For this outcome,
\[
\cov(\cF)\ge q_0+1
>
\frac{(1-\eta)\log N}{\lambda}
=
\frac{(1-\eta)\log N}{-\log(1-p)}.
\]
This proves \textup{(iii)} and completes the proof.
\end{proof}



\begin{lemma}\label{lem:attachment}
Fix constants
$0<\rho<1,$
$0<\eta<1,$
$K>0,$ and
$0<\alpha\le1$. 
There exists $k_0=k_0(\rho,\eta,K,\alpha)$ such that the following holds for every integer $k\ge k_0$. Let $B$ and $C$ be disjoint sets satisfying
\[
\frac{(1+\eta)k}{\rho}\le |B|\le Kk,
\qquad
k^\alpha\le |C|=s\le k.
\]
Then there is a bipartite graph with bipartition $(B,C)$ such that
\begin{enumerate}[label=\textup{(\roman*)},leftmargin=2.2em]
\item every vertex of $C$ has at least $k$ neighbors in $B$;
\item every set $X\subseteq B$ that dominates $C$ satisfies
$|X|>\frac{(1-\eta)\log s}{-\log(1-\rho)}$. 
\end{enumerate}
\end{lemma}

\begin{proof}
Let $m=|B|$ and
$s=|C|$. 
For every pair $(b,v)\in B\times C$, include the edge $bv$
independently with probability $\rho$. Let $I_{b,v}$ be the indicator
of the event that $bv$ is included. Thus all variables $I_{b,v}$ are
mutually independent and satisfy
$\Pr(I_{b,v}=1)=\rho$ and
$\Pr(I_{b,v}=0)=1-\rho$. 

Fix a vertex $v\in C$. Its number of neighbors in $B$ is
$d_B(v)=\sum_{b\in B} I_{b,v}$. 
The $m$ summands are independent Bernoulli random variables with
success probability $\rho$. Hence
$d_B(v)\sim\Bin(m,\rho),$ and
$\mathbb E d_B(v)=\rho m$. 
Let $\overline d=\rho m$. 
The hypothesis on $|B|$ gives
$\overline d\ge(1+\eta)k$. 
Define
$\theta=1-\frac{k}{\overline d}$. 
Then $0<\theta<1$, and
$
\theta
\ge
1-\frac{1}{1+\eta}
=
\frac{\eta}{1+\eta}.
$
Since
$
(1-\theta)\overline d=k,
$
the lower-tail estimate in Lemma~\ref{lem:chernoff} gives
\[
\Pr\bigl(d_B(v)<k\bigr)
\le \Pr\bigl(d_B(v)\le k\bigr)
=\Pr\bigl(d_B(v)\le(1-\theta)\overline d\bigr)
\le \exp\left(-\frac{\theta^2\overline d}{2}\right).
\]
Moreover,
\[
\theta^2\overline d
\ge
\left(\frac{\eta}{1+\eta}\right)^2(1+\eta)k
=
\frac{\eta^2}{1+\eta}k.
\]
Consequently,
\[
\Pr\bigl(d_B(v)<k\bigr)
\le
\exp\left(
-\frac{\eta^2}{2(1+\eta)}k
\right).
\]

There are $s$ vertices in $C$, and $s\le k$. Therefore, by the
union bound,
\begin{equation}\label{eq:attachment-degree-failure}
\Pr\bigl(\textup{(i) fails}\bigr)
\le
s\exp\left(
-\frac{\eta^2}{2(1+\eta)}k
\right)
\le
k\exp\left(
-\frac{\eta^2}{2(1+\eta)}k
\right).
\end{equation}

Let
$\lambda=-\log(1-\rho)>0,$
and $h_0=
\left\lfloor
\frac{(1-\eta)\log s}{\lambda}
\right\rfloor$. 
Since $s\le k$, we have
$
h_0\le
\frac{(1-\eta)\log k}{\lambda}
=O(\log k)$. 
On the other hand,
$m\ge\frac{(1+\eta)k}{\rho}$. 
It follows that $h_0<m$ for all sufficiently large $k$.

Fix a set $X\subseteq B$ with $|X|=r\le h_0$. For each fixed
$v\in C$, the $r$ potential edges between $v$ and $X$ are independently
absent with probability $1-\rho$. Hence
\[
\Pr\bigl(v\text{ has no neighbor in }X\bigr)=(1-\rho)^r.
\]
The definition of $h_0$ gives
$r\le h_0
\le \frac{(1-\eta)\log s}{\lambda},$
and hence
\[
(1-\rho)^r
=\exp(-\lambda r)
\ge
\exp\bigl(-(1-\eta)\log s\bigr)
=s^{-(1-\eta)}.
\]

For distinct vertices $v\in C$, the events
$
\{v\text{ has no neighbor in }X\}
$
are independent. Note that $X$ dominates $C$ if and only
if none of these events occurs. Therefore
\[
\Pr(X\text{ dominates }C)
=\left(1-(1-\rho)^r\right)^s
\le \exp\bigl(-s(1-\rho)^r\bigr)
\le \exp(-s^\eta),
\]
where the first inequality follows from $1-x\le e^{-x}$, and the
second follows from
$s(1-\rho)^r\ge
s\cdot s^{-(1-\eta)}
=s^\eta$. 

The number of subsets $X\subseteq B$ with $|X|\le h_0$ is at most
\[
\sum_{r=0}^{h_0}\binom{m}{r}
\le (h_0+1)m^{h_0}.
\]

Let $c_0=\frac{1-\eta}{\lambda}$. 
Then $h_0\le c_0\log k$. 
After increasing $k_0$ if necessary, we may assume that
$h_0+1\le k$ and
$m\le Kk\le k^2$. 
Consequently,
we have
\[
\begin{aligned}
(h_0+1)m^{h_0}
=\exp\bigl(\log(h_0+1)+h_0\log m\bigr)
\le
\exp\bigl(\log k+2c_0(\log k)^2\bigr)
\le
\exp\bigl(C(\log k)^2\bigr)
\end{aligned}
\]
for a constant $C=C(\rho,\eta)$.

By the definition of $h_0$, conclusion \textup{(ii)} fails precisely
when some set $X\subseteq B$ with $|X|\le h_0$ dominates $C$.
Therefore, using the union bound and the estimate above for each fixed
$X$, we obtain
\begin{align}
\Pr\bigl(\textup{(ii) fails}\bigr)
\le
\exp\bigl(C(\log k)^2-s^\eta\bigr)
\le
\exp\bigl(C(\log k)^2-k^{\alpha\eta}\bigr).
\label{eq:attachment-cover-failure}
\end{align}
Since $\alpha\eta>0$, the right-hand side tends to $0$ as
$k\to\infty$.

The right-hand side of
\eqref{eq:attachment-degree-failure} also tends to $0$. Hence, after
increasing $k_0=k_0(\rho,\eta,K,\alpha)$ if necessary, the sum of the
failure probabilities in
\eqref{eq:attachment-degree-failure} and
\eqref{eq:attachment-cover-failure} is smaller than $1$ for every
$k\ge k_0$. Thus, with positive probability, neither conclusion fails.
Any such outcome gives the required bipartite graph.
\end{proof}

\section{Proof of the lower bound in Theorem \ref{thm:main}}

Fix $\sigma>0$ and $\varepsilon>0$. 
Let $\varepsilon'=\min\{\varepsilon,\frac{1}{2}\}$. It suffices to prove the result with $\varepsilon'$ in place of $\varepsilon$. Relabeling $\varepsilon'$ as $\varepsilon$, we may assume $0<\varepsilon\le\frac{1}{2}$. We construct a graph that is used to derive the lower bound in Theorem \ref{thm:main}. 

\subsection{Constants used in the construction}

The next lemma records all constant choices in one place. Each condition is used later for one specific purpose.

\begin{lemma}\label{lem:constant-choice}
There exist constants $W$, $\delta$, and $\beta$ with
$W\ge\max\{6,2(1+\sigma)\},$
$0<\delta<\frac{1}{2},$ and 
$\beta>0,$
for which the following statements hold. Define
\[
a=\frac{1+3\delta}{(1-\delta)^2},
\qquad
\Psi_\delta(c)=\frac{(1-\delta)^2}{a}\,
\Phi\left(\frac{c}{a}\right),
\qquad
\mu=-\log\delta.
\]
Then
\begin{align}
&\Phi\left(\frac{W}{3}\right)>1-\frac{\varepsilon}{8},
\label{eq:C1}\\
&a<\min\left\{1+\frac{\sigma}{2},\frac{3}{2}\right\},
\label{eq:C2}\\
&(1+\sigma)\frac{1+\delta}{1+3\delta}>1,
\label{eq:C3}\\
&\left|\Psi_\delta(c)-\Phi(c)\right|<\frac{\varepsilon}{8}
\quad\text{for every }c\in[1+\sigma,W],
\label{eq:C4}\\
&\frac{(1-\delta)^2}{a}>1-\frac{\varepsilon}{8},
\label{eq:C5}\\
&\frac{4(1+\delta)\beta\mu}{1-\delta}<\delta.
\label{eq:C6}
\end{align}
\end{lemma}

\begin{proof}
We choose the constants in the order
$W, \delta, \beta$. 
First, by Lemma~\ref{lem:Phi},
$\Phi(x)\rightarrow 1$
as $x\rightarrow\infty$. 
Hence we may choose
$W\geq \max\{6,\,2(1+\sigma)\}$
so large that
$\Phi(\frac{W}{3})>1-\frac{\varepsilon}{8}$, which is  
\eqref{eq:C1}.

We next choose $\delta$. For clarity, write
$a(\delta):=\frac{1+3\delta}{(1-\delta)^2}$
and
$b(\delta):=\frac{(1-\delta)^2}{a(\delta)}
          =\frac{(1-\delta)^4}{1+3\delta}$. 
As $\delta\to 0$, we have
$a(\delta)\rightarrow 1$
and $b(\delta)\rightarrow 1$. 

Let
$A_*:=\min\left\{1+\frac{\sigma}{2},\,\frac{3}{2}\right\}$. 
Since $\sigma>0$, we have $A_*>1$. Therefore, because
$a(\delta)\to 1$, there exists $\delta_7>0$ such that for every
$0<\delta<\delta_7$, $a(\delta)<A_*$. 
So \eqref{eq:C2} is obeyed.

Define
$
h(\delta):=(1+\sigma)\frac{1+\delta}{1+3\delta}.
$
The function $h$ is continuous at $0$, and
$
h(0)=1+\sigma>1.
$
Hence there exists $\delta_8>0$ such that
for every $0<\delta<\delta_8$,
$(1+\sigma)\frac{1+\delta}{1+3\delta}>1$. 
Thus \eqref{eq:C3} holds for all sufficiently small $\delta$.

Choose
$0<\delta_0<\min\left\{\frac{1}{2},\delta_7\right\}$. 
For every $\delta\in[0,\delta_0]$, we have
$1\leq a(\delta)<A_*$. 
Consequently, for every $c\in[1+\sigma,W]$,
\[
\frac{1+\sigma}{A_*}
<
\frac{c}{a(\delta)}
\leq c
\leq W.
\]
Let
$m_*:=\frac{1+\sigma}{A_*}$. 
Since $A_*<1+\sigma$, we have $m_*>1$. Thus 
$c$, $\frac{c}{a(\delta)}\in I:=[m_*,W]\subset(1,\infty)$.
The function $\Phi$ is continuous on $(1,\infty)$, and hence it
is uniformly continuous on  $I$. Therefore
there exists $\tau>0$ such that, for all $x,y\in I$,
\[
|x-y|<\tau
\quad\Longrightarrow\quad
|\Phi(x)-\Phi(y)|<\frac{\varepsilon}{16}.
\]
Since $a(\delta)\to1$ and $b(\delta)\to1$, there exists
$\delta_9\in(0,\delta_0)$ such that, for all
$0<\delta<\delta_9$,
$W\left|\frac{1}{a(\delta)}-1\right|<\tau$
and
$|b(\delta)-1|<\frac{\varepsilon}{16}$.
For such a $\delta$ and every $c\in[1+\sigma,W]$, we have
\[
\left|\frac{c}{a(\delta)}-c\right|
=
c\left|\frac{1}{a(\delta)}-1\right|
\leq
W\left|\frac{1}{a(\delta)}-1\right|
<\tau.
\]
Hence, by the uniform continuity of $\Phi$ on $I$,
\[
\left|
\Phi\left(\frac{c}{a(\delta)}\right)-\Phi(c)
\right|
<
\frac{\varepsilon}{16}.
\]
As $0<\Phi(x)<1$ for $x>1$, we have
\begin{align*}
|\Psi_\delta(c)-\Phi(c)|
=
\left|
b(\delta)\Phi\left(\frac{c}{a(\delta)}\right)
-\Phi(c)
\right|
&\leq
|b(\delta)-1|
\Phi\left(\frac{c}{a(\delta)}\right)
+
\left|
\Phi\left(\frac{c}{a(\delta)}\right)-\Phi(c)
\right|\\
&<
\frac{\varepsilon}{16}
+
\frac{\varepsilon}{16}=
\frac{\varepsilon}{8},
\end{align*}
for every
$c\in[1+\sigma,W]$, and  so \eqref{eq:C4} holds.

As
$\frac{(1-\delta)^2}{a(\delta)}
=b(\delta)\rightarrow1$
as $\delta\to 0$, there exists $\delta_{10}>0$ such that
for every
$0<\delta<\delta_{10}$,
$\frac{(1-\delta)^2}{a(\delta)}
>1-\frac{\varepsilon}{8}$, so \eqref{eq:C5} holds.

Choose
$0<\delta<\min\left\{
\frac{1}{2},\,\delta_7,\,\delta_8,\,
\delta_9,\,\delta_{10}\right\}$. 
With this choice, \eqref{eq:C2}--\eqref{eq:C5}  are all obeyed.

Finally, fix this value of $\delta$ and let
$\mu:=-\log\delta$. 
Since $0<\delta<1$, we have $\mu>0$. Choose, for example,
$\beta:=\frac{\delta(1-\delta)}{8(1+\delta)\mu}$. 
Then $\beta>0$, and
$
\frac{4(1+\delta)\beta\mu}{1-\delta}
=\frac{\delta}{2}
<\delta$. 
Thus \eqref{eq:C6} holds.
\end{proof}


\subsection{Parameters}

Fix constants $W$, $\delta$, and $\beta$ satisfying
Lemma~\ref{lem:constant-choice}, and define $a$, $\mu$, and $\Psi_\delta$ accordingly. We assume throughout
this section that $k$ is sufficiently large in terms of
$\sigma$ and $\varepsilon$. The required lower bound on $k$ is
independent of $n$. Let
\[
\nu=\frac{n}{k+1}\ge1+\sigma
\text{ and }
x=\frac{n}{k}=\nu\left(1+\frac{1}{k}\right).
\]
Let
\[
R=
\begin{cases}
1,&x\le W,\\
\lceil \frac{x}{W}\rceil,&x>W
\end{cases}
\text{ and }
c=\frac{x}{R}.
\]
If $x\le W$, then $c=x\ge\nu\ge1+\sigma$. If $x>W$, then $c\le W$, so
\[
c=\frac{x}{\lceil \frac{x}{W}\rceil}
>\frac{x}{\frac{x}{W}+1}
=\frac{Wx}{x+W}
>\frac{W}{2}
\ge1+\sigma.
\]
In either case,
\begin{equation}\label{eq:c-range-new}
1+\sigma\le c\le W
\end{equation}
and $Rck=n$.
Let
\[
p=\frac{a}{c},
D_0=\lfloor\beta\log k\rfloor, 
N=\left\lceil\frac{D_0}{p}\right\rceil
\text{ and }
D=pN.
\]
Let
\[
\rho=1-\delta,
s=\lceil k^{\frac{1}{2}}\rceil \text{ and }
L=\left\lceil
\frac{(1+\delta)k}{\rho(1-\delta)D}
\right\rceil
=
\left\lceil
\frac{(1+\delta)k}{(1-\delta)^2D}
\right\rceil.
\]
By \eqref{eq:C2} and \eqref{eq:c-range-new},
\begin{equation}\label{eq:p-interval-new}
\frac{a}{W}\le p\le\frac{a}{1+\sigma}<1.
\end{equation}
Thus $p\in [\frac{a}{W}, \frac{a}{1+\sigma}]\subset (0,1)$. From the definition of $N$,
\begin{equation}\label{eq:D-bounds-new}
D_0\le D<D_0+1.
\end{equation}
Consequently, there are positive constants $b_1,b_2,b_3$, depending only on $\sigma$ and $\varepsilon$, such that for all sufficiently large $k$,
\begin{equation}\label{eq:N-bounds-new}
b_1\log k\le N\le b_2\log k \text{ and }
\left|\log N-\log\log(k+1)\right|\le b_3.
\end{equation}
The first pair of inequalities follows directly from \eqref{eq:p-interval-new} and $D_0\sim\beta\log k$; the second follows by taking logarithms.

Let $H=N(L+s)$.  We give the following estimate.

\begin{lemma}\label{lem:block-budget}
For all sufficiently large $k$,
$H\le ck$ and $NL\ge k+1$. 
The required lower bound on $k$ depends only on $\sigma$ and $\varepsilon$.
\end{lemma}

\begin{proof}
Since $D=pN$ and $p=\frac{a}{c}$, the definition of $L$ gives
\begin{equation}
NL\le N\left(\frac{(1+\delta)k}{(1-\delta)^2D}+1\right)
=ck\frac{1+\delta}{1+3\delta}+N.
\label{eq:NL-upper-new}
\end{equation}
Let
$\theta=1-\frac{1+\delta}{1+3\delta}
=\frac{2\delta}{1+3\delta}>0$. 
By \eqref{eq:N-bounds-new} and $s=\lceil k^{\frac{1}{2}}\rceil$,
\[
\frac{N(s+1)}{k}\le
\frac{b_2\log k\,(k^{\frac{1}{2}}+2)}{k}\longrightarrow0.
\]
Hence, after increasing the lower bound on $k$,
$N(s+1)\le\theta(1+\sigma)k\le\theta ck$. 
Combining this inequality with \eqref{eq:NL-upper-new} yields
$H=NL+Ns\le ck$. 

For the lower bound, the definition of $L$ gives
$NL\ge ck\frac{1+\delta}{1+3\delta}$. 
By \eqref{eq:C3} and $c\ge1+\sigma$, the right-hand side is larger than $k$. Since $NL$ is an integer, $NL\ge k+1$.
\end{proof}

\subsection{Building the graph}

By \eqref{eq:p-interval-new} and \eqref{eq:N-bounds-new},
Lemma~\ref{lem:set-system} applies for all sufficiently large $k$,
with
$p_-=\frac{a}{W},$
$p_+=\frac{a}{1+\sigma},$
$\eta=\delta,$
and with the present values of $N$ and $p$.
Take $R$ disjoint copies
$\cF_i=(U_i;F_{i,1},\ldots,F_{i,N})$
$(1\le i\le R)$
of the resulting set system. Let
$U=\bigcup_{i=1}^R U_i,$ and
$t=|U|=RN$. 

For every $u\in U$, create a clique $C_u$ of order $s$, which we call a satellite clique. The cliques $C_u$ are pairwise disjoint and pairwise anticomplete. For each indexed set $F_{i,j}$, create $L$ vertices
$a_{i,j,1},\ldots,a_{i,j,L},$
each carrying the label $F_{i,j}$. Let $A_0$ be the set of all these vertices. If $u\in U_i$, define
\[
B_u=\{a_{i,j,r}\in A_0:u\in F_{i,j}\}.
\]
Because each label is repeated exactly $L$ times,
$
|B_u|=L d_{\cF_i}(u).
$

The degree bounds in Lemma~\ref{lem:set-system} imply
$
|B_u|\ge(1-\delta)DL
\ge\frac{(1+\delta)k}{\rho}.
$
They also give
$
|B_u|
\le(1+\delta)DL
\le\frac{(1+\delta)^2}{(1-\delta)^2}k+(1+\delta)D.
$
By \eqref{eq:D-bounds-new}, $D\le\beta\log k+1$. Therefore there is a constant
$
K=2+\frac{(1+\delta)^2}{(1-\delta)^2}
$
such that $|B_u|\le Kk$ for all sufficiently large $k$ and every $u$.

Apply Lemma~\ref{lem:attachment} to every pair $(B_u,C_u)$ with
$\eta=\delta,$
$\rho=1-\delta,$
$\alpha=\frac{1}{2},$
and the constant $K$ above. For each $u\in U$, choose one such bipartite graph on
$(B_u,C_u)$ and add all of its edges. Add no other edges between $A_0$ and the satellite cliques.

Now, there are
$RNL+RNs=RH\le Rck=n$
vertices. Add $n-RH$ vertices, let $A_1$ be their set, put
$A=A_0\cup A_1,$
and make $A$ a clique. No vertex of $A_1$ is adjacent to a satellite
clique. This completes the construction of $G$. The construction is
illustrated schematically in
Figure~\ref{fig:satellite-attachment-schematic}.

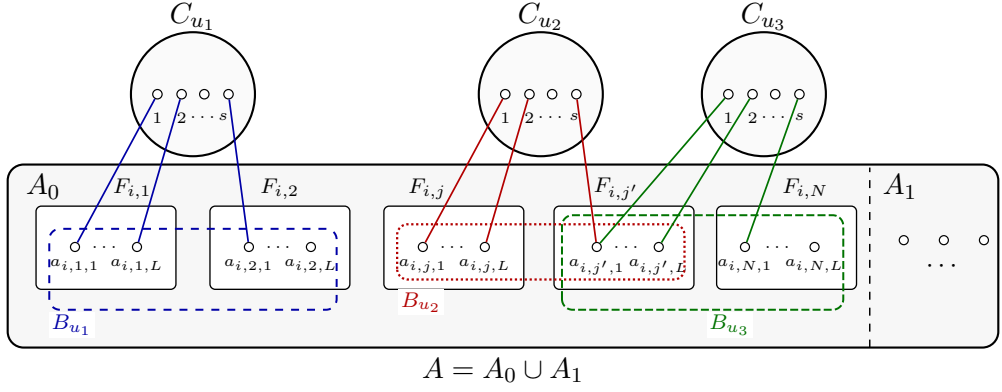
\begin{figure}[htbp]
\centering
\begin{tikzpicture}[
  x=1cm,
  y=1cm,
  block/.style={
    draw=black,
    rounded corners=2pt,
    minimum width=1.85cm,
    minimum height=1.10cm,
    fill=white,
    line width=0.5pt
  },
  avertex/.style={
    circle,
    draw=black,
    fill=white,
    inner sep=1.2pt,
    line width=0.45pt
  },
  cvertex/.style={
    circle,
    draw=black,
    fill=white,
    inner sep=1.25pt,
    line width=0.45pt
  }
]

\draw[
  rounded corners=6pt,
  line width=0.8pt,
  fill=black!3
]
  (-6.55,-1.3) rectangle (6.55,1.12);

\draw[
  dashed,
  line width=0.55pt
]
  (4.85,-1.3) -- (4.85,1.08);

\node[font=\bfseries] at (-6.08,0.83) {$A_0$};
\node[font=\bfseries] at (5.25,0.83) {$A_1$};
\node[font=\bfseries] at (0,-1.6) {$A=A_0\cup A_1$};

\node[block] at (-5.25,0.02) {};
\node[block] at (-2.95,0.02) {};
\node[block] at (-0.65,0.02) {};
\node[block] at ( 1.6,0.02) {};
\node[block] at ( 3.75,0.02) {};

\node[font=\scriptsize\bfseries]
  at (-4.9,0.8) {$F_{i,1}$};

\node[font=\scriptsize\bfseries]
  at (-2.95,0.8) {$F_{i,2}$};

\node[font=\scriptsize\bfseries]
  at (-1,0.8) {$F_{i,j}$};

\node[font=\scriptsize\bfseries]
  at (1.5,0.8) {$F_{i,j'}$};

\node[font=\scriptsize\bfseries]
  at (4,0.8) {$F_{i,N}$};

\node[avertex] (a11) at (-5.66,0.03) {};
\node[avertex] (a1L) at (-4.84,0.03) {};

\node[font=\tiny]
  at (-5.25,0.03) {$\cdots$};

\node[font=\tiny,anchor=north]
  at (-5.66,-0) {$a_{i,1,1}$};

\node[font=\tiny,anchor=north]
  at (-4.84,-0) {$a_{i,1,L}$};

\node[avertex] (a21) at (-3.36,0.03) {};
\node[avertex] (a2L) at (-2.54,0.03) {};

\node[font=\tiny]
  at (-2.95,0.03) {$\cdots$};

\node[font=\tiny,anchor=north]
  at (-3.36,-0) {$a_{i,2,1}$};

\node[font=\tiny,anchor=north]
  at (-2.54,-0) {$a_{i,2,L}$};

\node[avertex] (aj1) at (-1.06,0.03) {};
\node[avertex] (ajL) at (-0.24,0.03) {};

\node[font=\tiny]
  at (-0.65,0.03) {$\cdots$};

\node[font=\tiny,anchor=north]
  at (-1.06,-0) {$a_{i,j,1}$};

\node[font=\tiny,anchor=north]
  at (-0.24,-0) {$a_{i,j,L}$};

\node[avertex] (am1) at (1.24,0.03) {};
\node[avertex] (amL) at (2.06,0.03) {};

\node[font=\tiny]
  at (1.65,0.03) {$\cdots$};

\node[font=\tiny,anchor=north]
  at (1.24,-0) {$a_{i,j',1}$};

\node[font=\tiny,anchor=north]
  at (2.06,-0) {$a_{i,j',L}$};

\node[avertex] (aN1) at (3.2,0.03) {};
\node[avertex] (aNL) at (4.12,0.03) {};

\node[font=\tiny]
  at (3.65,0.03) {$\cdots$};

\node[font=\tiny,anchor=north]
  at (3.2,-0) {$a_{i,N,1}$};

\node[font=\tiny,anchor=north]
  at (4.12,-0) {$a_{i,N,L}$};

\node[avertex] at (5.30,0.12) {};
\node[avertex] at (5.83,0.12) {};
\node[avertex] at (6.36,0.12) {};

\node[font=\small]
  at (5.83,-0.22) {$\cdots$};

\draw[
  rounded corners=4pt,
  dashed,
  draw=blue!65!black,
  line width=0.75pt
]
  (-6,-0.8) rectangle (-2.2,0.27);

\node[
  font=\scriptsize\bfseries,
  text=blue!65!black,
  fill=white,
  inner sep=1pt
]
  at (-5.70,-1) {$B_{u_1}$};

\draw[
  rounded corners=4pt,
  densely dotted,
  draw=red!70!black,
  line width=0.85pt
]
  (-1.4,-0.4) rectangle (2.4,0.33);

\node[
  font=\scriptsize\bfseries,
  text=red!70!black,
  fill=white,
  inner sep=1pt
]
  at (-1.08,-0.7) {$B_{u_2}$};

\draw[
  rounded corners=4pt,
  dash pattern=on 3pt off 1.2pt,
  draw=green!45!black,
  line width=0.75pt
]
  (0.78,-0.8) rectangle (4.5,0.45);

\node[
  font=\scriptsize\bfseries,
  text=green!45!black,
  fill=white,
  inner sep=1pt
]
  at (3,-1) {$B_{u_3}$};

\draw[
  line width=0.8pt,
  fill=black!2
]
  (-4.10,2.05) circle (0.82);

\draw[
  line width=0.8pt,
  fill=black!2
]
  (0.50,2.05) circle (0.82);

\draw[
  line width=0.8pt,
  fill=black!2
]
  (3.45,2.05) circle (0.82);

\node[font=\bfseries]
  at (-4.10,3.08) {$C_{u_1}$};

\node[font=\bfseries]
  at (0.50,3.08) {$C_{u_2}$};

\node[font=\bfseries]
  at (3.45,3.08) {$C_{u_3}$};

\node[cvertex] (c11) at (-4.57,2.05) {};
\node[cvertex] (c12) at (-4.25,2.05) {};
\node[cvertex] (c13) at (-3.95,2.05) {};
\node[cvertex] (c1s) at (-3.63,2.05) {};

\node[font=\tiny,anchor=north]
  at (-4.57,1.94) {$1$};

\node[font=\tiny,anchor=north]
  at (-4.25,1.94) {$2$};

\node[font=\tiny,anchor=north]
  at (-3.95,1.9) {$\dots$};

\node[font=\tiny,anchor=north]
  at (-3.7,1.94) {$s$};

\node[cvertex] (c21) at (0.03,2.05) {};
\node[cvertex] (c22) at (0.35,2.05) {};
\node[cvertex] (c23) at (0.65,2.05) {};
\node[cvertex] (c2s) at (0.97,2.05) {};

\node[font=\tiny,anchor=north]
  at (0.03,1.94) {$1$};

\node[font=\tiny,anchor=north]
  at (0.35,1.94) {$2$};

\node[font=\tiny,anchor=north]
  at (0.65,1.9) {$\dots$};

\node[font=\tiny,anchor=north]
  at (0.94,1.94) {$s$};

\node[cvertex] (c31) at (2.98,2.05) {};
\node[cvertex] (c32) at (3.30,2.05) {};
\node[cvertex] (c33) at (3.60,2.05) {};
\node[cvertex] (c3s) at (3.92,2.05) {};

\node[font=\tiny,anchor=north]
  at (2.98,1.94) {$1$};

\node[font=\tiny,anchor=north]
  at (3.30,1.94) {$2$};

\node[font=\tiny,anchor=north]
  at (3.60,1.9) {$\dots$};

\node[font=\tiny,anchor=north]
  at (3.92,1.94) {$s$};

\draw[line width=0.65pt,blue!65!black]
  (a11) -- (c11);

\draw[line width=0.65pt,blue!65!black]
  (a1L) -- (c12);

\draw[line width=0.65pt,blue!65!black]
  (a21) -- (c1s);

\draw[line width=0.65pt,red!70!black]
  (aj1) -- (c21);

\draw[line width=0.65pt,red!70!black]
  (ajL) -- (c22);

\draw[line width=0.65pt,red!70!black]
  (am1) -- (c2s);

\draw[line width=0.65pt,green!45!black]
  (am1) -- (c31);

\draw[line width=0.65pt,green!45!black]
  (amL) -- (c32);

\draw[line width=0.65pt,green!45!black]
  (aN1) -- (c3s);

\end{tikzpicture}

\caption{Schematic of the construction of $G$.}
\label{fig:satellite-attachment-schematic}
\end{figure}

\begin{proposition}\label{prop:admissible}
The graph $G$ is a connected $n$-vertex graph with $\delta(G)\ge k$.
\end{proposition}

\begin{proof}
By construction, $|V(G)|=RH+(n-RH)=n$. By
Lemma~\ref{lem:attachment}\textup{(i)}, every vertex outside $A$
has at least $k$ neighbors in $A$, and hence has degree at least
$k$. On the other hand, Lemma~\ref{lem:block-budget} gives
$
|A|\ge |A_0|=RNL\ge NL\ge k+1.
$
Since $A$ is a clique, every vertex of $A$ has degree at least
$|A|-1\ge k$. Thus $\delta(G)\ge k$. Furthermore, $A$ is connected
and every vertex outside $A$ has a neighbor in $A$, so $G$ is
connected.
\end{proof}

\begin{lemma}\label{lem:small-dominating-set}
The graph $G$ satisfies
$\gamma(G)\le t+1$. 
\end{lemma}

\begin{proof}
For each $u\in U$, choose a vertex $c_u\in C_u$, and choose a
vertex $v\in A$. Let
$
S=\{v\}\cup\{c_u:u\in U\}.
$
Since $A$ and each $C_u$ are cliques, every vertex of $A\setminus S$
is adjacent to $v$, and every vertex of $C_u\setminus S$ is adjacent
to $c_u$. Thus $S$ is a dominating set of $G$. Since $|U|=t$, we have
$|S|=t+1$, and hence
$
\gamma(G)\le t+1.
$
\end{proof}

\subsection{Estimate the connected domination number}

The following proposition gives the key lower bound on the size of a connected dominating set.

\begin{proposition}\label{prop:key-count}
Every connected dominating set $S$ of $G$ satisfies
$
|S|>
 t+(1-\delta)^2R
\frac{\log N}{-\log(1-p)}.
$
\end{proposition}

\begin{proof}
Let
$X=S\cap A_0,$ and $q=|X|$. 
For each $i\in[R]$, let
\[
J_i
=
\left\{
j\in[N]:
\text{there exists }\ell\in[L]\text{ such that }
a_{i,j,\ell}\in X
\right\}.
\]
Suppose that the sets indexed by $J_i$ do not cover $U_i$, and choose
$
u\in U_i\setminus\bigcup_{j\in J_i}F_{i,j}.
$
Then $X\cap B_u=\emptyset$. If $S\cap C_u=\emptyset$, then, since the only neighbors of $C_u$ outside itself lie in $B_u$, the set $S$ does not dominate $C_u$, a contradiction.
Suppose $S\cap C_u\ne\emptyset$. 
Since \(N \ge 2\) for all sufficiently large \(k\), the construction
contains a satellite clique \(C_{u'}\) distinct from \(C_u\). The satellite
cliques are pairwise anticomplete, so a set contained entirely in \(C_u\)
cannot dominate \(C_{u'}\). Thus \(S \setminus C_u \neq \emptyset\). Since
\(X \cap B_u = \emptyset\), there is no edge between
\(S \cap C_u\) and \(S \setminus C_u\), again a contradiction.
Therefore the labels indexed by $J_i$ cover $U_i$. Lemma~\ref{lem:set-system}(iii) gives
$
|J_i|>
(1-\delta)\frac{\log N}{-\log(1-p)}.
$
Since each vertex of $X$ has a unique pair of indices $(i,j)$ and every $j\in J_i$ requires at least one vertex $a_{i,j,\ell}\in X$, we have
\begin{equation}\label{eq:q-cover-new}
q\ge\sum_{i=1}^R|J_i|
>
R(1-\delta)\frac{\log N}{-\log(1-p)}.
\end{equation}

Let
$Z=\{u\in U:S\cap C_u=\emptyset\},$
and let $r=|Z|$. 
For $u\in Z$, the set $X\cap B_u$ must dominate $C_u$. Lemma~\ref{lem:attachment}(ii), together with $-\log(1-\rho)=\mu$, gives
$
|X\cap B_u|>
\frac{(1-\delta)\log s}{\mu}.
$
Summing over $u\in Z$,
\begin{equation}\label{eq:lower-double-new}
r\frac{(1-\delta)\log s}{\mu}
\le \sum_{u\in Z}|X\cap B_u|.
\end{equation}
For every $v\in X$, write $v=a_{i,j,\ell}$ for some
$i\in[R]$, $j\in[N]$, and $\ell\in[L]$. Then
\[
\bigl|\{u\in Z:v\in B_u\}\bigr|
=
|Z\cap F_{i,j}|
\le |F_{i,j}|
\le (1+\delta)D.
\]
Therefore, double-counting the pairs $(v,u)$ with $u\in Z$ and $v\in X\cap B_u$ gives
\begin{equation}\label{eq:upper-double-new}
\sum_{u\in Z}|X\cap B_u|=\sum_{v\in X}\bigl|\{u\in Z:v\in B_u\}\bigr|
\le(1+\delta)Dq.
\end{equation}
By \eqref{eq:D-bounds-new}, $D<D_0+1\le2\beta\log k$ for all sufficiently large $k$, while $s=\lceil k^{\frac{1}{2}}\rceil$ gives $\log s\ge\frac{\log k}{2}$. 
Combining \eqref{eq:lower-double-new} and
\eqref{eq:upper-double-new}, and using
\eqref{eq:C6}, we obtain
\[
r
\le
\frac{(1+\delta)D\mu}{(1-\delta)\log s}\,q
\le
\frac{4(1+\delta)\beta\mu}{1-\delta}\,q
<\delta q.
\]

Exactly $t-r$ satellite cliques meet $S$, and each contributes at least one vertex outside $X$. Therefore
\[
|S|\ge q+t-r>t+(1-\delta)q.
\]
Substituting \eqref{eq:q-cover-new} proves the proposition.
\end{proof}

\subsection{Proof of the lower bound in Theorem \ref{thm:main}}


\begin{proof}[Proof of the lower bound in Theorem \ref{thm:main}]
Let
$L_k=\log\log(k+1),$
and $\lambda=-\log(1-p)$. 
By Proposition~\ref{prop:admissible}, the graph $G$ is a connected $n$-vertex graph with $\delta(G)\ge k$. Hence Lemma~\ref{lem:small-dominating-set}
and Proposition~\ref{prop:key-count} imply
\[
M(n,k)>
(1-\delta)^2R\frac{\log N}{\lambda}-1.
\]
Since $p=\frac{a}{c}$,
we have $\lambda=\log\frac{c}{c-a}$ and
$\frac{1}{c\lambda}=\frac{1}{a}\Phi\left(\frac{c}{a}\right)$.
Recall that  $R=\frac{n}{ck}$ and $\nu=\frac{n}{k+1}$. Note that $\frac{c}{a}>1$ and $0<\frac{(1-\delta)^2}{a}<1$. By Lemma~\ref{lem:Phi}, $0<\Psi_\delta(c)<1$. Then 
\begin{align*}
\frac{M(n,k)}{\nu L_k}
&>
\frac{k+1}{k}\,
\Psi_\delta(c)\,
\frac{\log N}{L_k}
-\frac{1}{\nu L_k}\\
&\ge \Psi_\delta(c)-
\left|
\frac{k+1}{k}\cdot \frac{\log N}{L_k}-1
\right|\Psi_\delta(c)
-\frac{1}{\nu L_k}.
\end{align*}
By \eqref{eq:N-bounds-new}, we have
$
|\log N-L_k|\le b_3$,
so
\[
\left|
\frac{k+1}{k}\cdot \frac{\log N}{L_k}-1
\right|
=
\left|
\left(1+\frac{1}{k}\right)
\frac{\log N-L_k}{L_k}
+\frac{1}{k}
\right|
\le
\frac{1}{k}
+
\left(1+\frac{1}{k}\right)\frac{b_3}{L_k}.
\]
Hence
$
\frac{k+1}{k}\cdot \frac{\log N}{L_k}\rightarrow1,
$
and the difference from $1$ is bounded by an expression depending only on $k$, $\sigma$, and $\varepsilon$, not on $n$. Also $\frac{1}{\nu L_k}\le\frac{1}{(1+\sigma)L_k}$. 
Since $0<\Psi_\delta(c)<1$,
\[
\Psi_\delta(c)-
\left|
\frac{k+1}{k}\cdot \frac{\log N}{L_k}-1
\right|\Psi_\delta(c)
-\frac{1}{\nu L_k}
\ge
\Psi_\delta(c)
-
\left(
\frac{1}{k}
+
\left(1+\frac{1}{k}\right)\frac{b_3}{L_k}
+
\frac{1}{(1+\sigma)L_k}
\right).
\]
Evidently, $\frac{1}{k}
+
\left(1+\frac{1}{k}\right)\frac{b_3}{L_k}
+
\frac{1}{(1+\sigma)L_k}\to 0$
 as $k\to\infty$.
Hence, by increasing $k_0=k_0(\sigma,\varepsilon)$ if necessary,
we may ensure that, for every \(k\ge k_0\),
$
\frac{1}{k}
+
\left(1+\frac{1}{k}\right)\frac{b_3}{L_k}
+
\frac{1}{(1+\sigma)L_k}
\le \frac{\varepsilon}{2}.
$
Consequently, 
\begin{equation}\label{eq:normalized-simple}
\frac{M(n,k)}{\nu L_k}
>
\Psi_\delta(c)-\frac{\varepsilon}{2}.
\end{equation}
Next, we show that \begin{equation}\label{bo}
\Psi_\delta(c)>\Phi(\nu)-\frac{\varepsilon}{4}.
\end{equation}
 There are two cases.
If $x=\frac{n}{k}\le W$, then $R=1$ and
$
c=x=\nu\left(1+\frac{1}{k}\right)\ge\nu$.
So we have by \eqref{eq:C4} and the monotonicity of $\Phi$ that
\[
\Psi_\delta(c)
>\Phi(c)-\frac{\varepsilon}{8}
\ge\Phi(\nu)-\frac{\varepsilon}{8}>\Phi(\nu)-\frac{\varepsilon}{4}.
\]
If $x>W$, then \(c>\frac{W}{2}\). Since \(a<\frac{3}{2}\), we have
$
\frac{c}{a}>\frac{W}{3}.
$
Hence, by the monotonicity of \(\Phi\) and \eqref{eq:C1},
$
\Phi\left(\frac{c}{a}\right)
>
\Phi\left(\frac{W}{3}\right)
>
1-\frac{\varepsilon}{8}.
$
Combining this with \eqref{eq:C5}, we have
\[
\Psi_\delta(c)
=
\frac{(1-\delta)^2}{a}
\Phi\left(\frac{c}{a}\right)
>
\left(1-\frac{\varepsilon}{8}\right)^2
>1-\frac{\varepsilon}{4}>\Phi(\nu)-\frac{\varepsilon}{4}.
\]
Thus, in either case, we have \eqref{bo}.
Combining \eqref{eq:normalized-simple} and \eqref{bo}, we have
\[
\frac{M(n,k)}{\nu L_k}-\Phi(\nu)
>-\frac{3\varepsilon}{4}
>-\varepsilon,
\]
as desired.
\end{proof}

\section{Proofs of Theorem~\ref{thm:full-range} and Corollaries~\ref{cor:theta} and~\ref{cor:large-ratio}}

\begin{proof}[Proof of Theorem~\ref{thm:full-range}]
Let
$L_k=\log\log(k+1)$
and
$A(n,k)=\frac{M(n,k)}{\nu L_k}$. Fix $\eta>0$.

By Lemma~\ref{lem:Phi}, $\Phi(x)\to0$ as $x\to1^+$.
Hence we may choose $\sigma_0>0$, depending only on $\eta$,
such that
$
\Phi(1+\sigma_0)<\frac{\eta}{2}.
$

Applying Theorem~\ref{thm:main} with
$\sigma=\sigma_0$ and  $\varepsilon=\frac{\eta}{2}$, we obtain an
integer $k_1=k_1(\sigma_0,\eta)$ such that
$
\left|A(n,k)-\Phi(\nu)\right|
\le \frac{\eta}{2}
$
for all $k\ge k_1$ and $\nu\ge1+\sigma_0$.

Assume that
$
1<\nu<1+\sigma_0.
$
By Proposition~\ref{prop:upper},
\[
M(n,k)
\le
\frac{\log \nu+L_k}
{\log \frac{\nu}{\nu-1} }
+2\nu+1.
\]
As
$
\Phi(\nu)
=
\frac{1}
{\nu\log\frac{\nu}{\nu-1}}$ and $\nu>1$,
\begin{align*}
A(n,k)
&\le\frac{\log \nu+L_k}
{\nu L_k\log \frac{\nu}{\nu-1} }
+\frac{2\nu+1}{\nu L_k}\\
&=
\Phi(\nu)
\left(
1+\frac{\log\nu}{L_k}
\right)
+
\frac{2+\frac{1}{\nu}}{L_k}
\le
\Phi(\nu)
\left(
1+\frac{\log\nu}{L_k}
\right)
+
\frac{3}{L_k}.
\end{align*}
Trivially, 
$\gamma_c(H)\ge\gamma(H)$ for every connected graph $H$, so  $A(n,k)\ge0$. It then follows that
\[
-\Phi(\nu)
\le
A(n,k)-\Phi(\nu)
\le
\frac{\Phi(\nu)\log\nu+3}{L_k}.
\]
Consequently,
\[
\left|A(n,k)-\Phi(\nu)\right|
\le
\max\left\{
\Phi(\nu),
\frac{\Phi(\nu)\log\nu+3}{L_k}
\right\}.
\]
Since $\Phi$ is increasing and
$1<\nu<1+\sigma_0$, we have
$\Phi(\nu)\le\Phi(1+\sigma_0)<\frac{\eta}{2}$,
and
\[
\frac{\Phi(\nu)\log\nu+3}{L_k}\le \frac{
\Phi(1+\sigma_0)\log(1+\sigma_0)+3
}{L_k}.
\]
Since $L_k\to\infty$, there exists an
integer $k_2=k_2(\sigma_0,\eta)$ such that
\[
\frac{
\Phi(1+\sigma_0)\log(1+\sigma_0)+3
}{L_k}
<
\frac{\eta}{2}
\]
for all $k\ge k_2$.
Therefore,
\[
\left|A(n,k)-\Phi(\nu)\right|
\le
\max\left\{
\Phi(1+\sigma_0),
\frac{
\Phi(1+\sigma_0)\log(1+\sigma_0)+3
}{L_k}
\right\}<\frac{\eta}{2}
\]
for all $k\ge k_2$ and $1<\nu<1+\sigma_0$.

The number $\sigma_0$ was chosen only in terms of $\eta$.
Therefore, both $k_1$ and $k_2$ depend only on $\eta$ and
not on $n$. Let
$
K=\max\{k_1,k_2,3\}.
$
Let $k\ge K$ and let $n>k+1$ be an integer. Then $\nu>1$.
If $\nu\ge1+\sigma_0$, the estimate obtained from
Theorem~\ref{thm:main} applies. Otherwise,
$1<\nu<1+\sigma_0$, and the preceding estimate applies.
In either case,
\[
\left|
\frac{M(n,k)}
{\nu\log\log(k+1)}
-
\Phi(\nu)
\right|
<\eta.
\]
This completes the proof.
\end{proof}

\begin{proof}[Proof of Corollary~\ref{cor:theta}]
Fix $C>1$. By Lemma~\ref{lem:Phi},
$
0<\Phi(C)<1.
$
Let
$
\eta=\frac{\Phi(C)}{2}$ and $\nu=\frac{n}{k+1}$.
By Theorem~\ref{thm:full-range}, there exists an integer
$k_0=k_0(C)$ such that, for all integers $k\ge k_0$ and
$n>k+1$,
\[
\left|
\frac{M(n,k)}
{\nu\log\log(k+1)}
-
\Phi(\nu)
\right|
<
\frac{\Phi(C)}{2}.
\]

Now suppose that $n\ge C(k+1)$. Then $\nu\ge C$, so the
monotonicity of $\Phi$ gives
$
\Phi(\nu)\ge\Phi(C).
$
Consequently,
\[
\frac{M(n,k)}
{\nu\log\log(k+1)}
>
\Phi(\nu)-\frac{\Phi(C)}{2}
\ge
\frac{\Phi(C)}{2},
\]
so
\[
M(n,k)>\frac{\Phi(C)}{2} \nu\log\log(k+1).
\]
On the other hand, since $\Phi(\nu)<1$ and $\Phi(C)<1$, we have
\[
\frac{M(n,k)}
{\nu\log\log(k+1)}
<
\Phi(\nu)+\frac{\Phi(C)}{2}
<
\frac{3}{2},
\]
so
\[
M(n,k)
<
\frac{3}{2}\nu \log\log(k+1).
\]
This completes the proof.
\end{proof}

\begin{proof}[Proof of Corollary~\ref{cor:large-ratio}]
Let
$
\nu=\frac{n}{k+1}.
$
Since $\nu\to\infty$, we have $n>k+1$ for all sufficiently
large $k$. By Theorem~\ref{thm:full-range},
\[
\frac{M(n,k)}
{\nu\log\log(k+1)}
=
\Phi(\nu)+o(1).
\]
Lemma~\ref{lem:Phi} gives $\Phi(\nu)\to1$. Therefore,
$
\frac{M(n,k)}
{\nu\log\log(k+1)}
\to1.
$
Substituting $\nu=\frac{n}{k+1}$ gives
$
M(n,k)
\sim
\frac{n}{k+1}\log\log(k+1).
$
\end{proof}

\medskip

\noindent {\bf Acknowledgements.}
This work was supported by the
National Natural Science Foundation of China (No.~12571364).

\medskip
\noindent\textbf{Declaration of competing interest}

\noindent There is no competing interest.

\medskip
\noindent\textbf{Data availability}

\noindent There is no data associated with this paper.

\end{document}